\documentclass[11pt]{article}
\usepackage{amsthm,geometry,amssymb,amsmath,enumerate,float,tikz,cite,algorithm2e,setspace}
\newtheorem{theorem}{Theorem}[section]

\usepackage{showlabels}
\usepackage{lineno}

\title{Girth, minimum degree, independence, and
broadcast independence}
\author{S. Bessy$^1$ \and  D. Rautenbach$^2$}
\date{}

\begin{document}
\onehalfspace

\maketitle
\vspace{-10mm}
\begin{center}
{\small
$^1$ 
Laboratoire d'Informatique, de Robotique et de Micro\'{e}lectronique de Montpellier,\\
Montpellier, France, \texttt{stephane.bessy@lirmm.fr}\\[3mm]
$^2$ Institute of Optimization and Operations Research, Ulm University,\\
Ulm, Germany, \texttt{dieter.rautenbach@uni-ulm.de}}
\end{center}

\begin{abstract}
An independent broadcast on a connected graph $G$
is a function $f:V(G)\to \mathbb{N}_0$
such that, for every vertex $x$ of $G$, 
the value $f(x)$ is at most the eccentricity of $x$ in $G$,
and $f(x)>0$ implies that $f(y)=0$ 
for every vertex $y$ of $G$ within distance at most $f(x)$ from $x$.
The broadcast independence number $\alpha_b(G)$ of $G$
is the largest weight $\sum\limits_{x\in V(G)}f(x)$
of an independent broadcast $f$ on $G$.

It is known that $\alpha(G)\leq \alpha_b(G)\leq 4\alpha(G)$
for every connected graph $G$,
where $\alpha(G)$ is the independence number of $G$.
If $G$ has girth $g$ and minimum degree $\delta$,
we show that 
$\alpha_b(G)\leq 2\alpha(G)$
provided that 
$g\geq 6$ and $\delta\geq 3$
or that $g\geq 4$ and $\delta\geq 5$.
Furthermore, 
we show that, 
for every positive integer $k$,
there is a connected graph $G$ of girth at least $k$ and minimum degree at least $k$ 
such that 
$\alpha_b(G)\geq 2\left(1-\frac{1}{k}\right)\alpha(G)$.
Our results imply that 
lower bounds on the girth and the minimum degree
of a connected graph $G$
can lower the fraction $\frac{\alpha_b(G)}{\alpha(G)}$
from $4$ below $2$, but not any further.
\end{abstract}
{\small 
\begin{tabular}{lp{13cm}}
{\bf Keywords}: & broadcast independence; independence; packing\\
{\bf MSC 2010}: & 05C69; 05C12
\end{tabular}
}

\pagebreak

\section{Introduction}

In the present paper, we relate broadcast independence 
to independence and packings in graphs
of large girth and minimum degree.
We consider finite, simple, and undirected graphs,
and use standard terminology and notation.
A set $I$ of pairwise nonadjacent vertices of a graph $G$ 
is an {\it independent set} in $G$,
and the maximum cardinality of an independent set in $G$
is the {\it independence number} $\alpha(G)$ of $G$.
Similarly, a set $P$ of vertices of $G$ is a {\it packing}
if ${\rm dist}_G(x,y)\geq 3$ 
for every two distinct vertices $x$ and $y$ in $P$, 
where ${\rm dist}_G(x,y)$ is the distance of $x$ and $y$ in $G$.
The maximum cardinality of a packing in $G$
is the {\it packing number} $\rho(G)$ of $G$.
The independence number and the packing number 
are among the most fundamental and 
well studied graph parameters \cite{tovo}.
Broadcast independence was introduced by Erwin \cite{er},
cf. also \cite{duerhahehe}, 
and was studied in \cite{ahboso,boze,bera1,bera2}.
Let $\mathbb{N}_0$ be the set of nonnegative integers.
For a connected graph $G$, a function $f:V(G)\to \mathbb{N}_0$ is an {\it independent broadcast on $G$} if
\begin{quote}
\begin{enumerate}[(B1)]
\item $f(x)\leq {\rm ecc}_G(x)$ for every vertex $x$ of $G$, where ${\rm ecc}_G(x)$ is the eccentricity of $x$ in $G$, and
\item ${\rm dist}_G(x,y)>\max\{ f(x),f(y)\}$ for every two distinct vertices $x$ and $y$ of $G$ with $f(x),f(y)>0$.
\end{enumerate}
\end{quote}
The {\it weight} of $f$ is $\sum\limits_{x\in V(G)}f(x)$.
The {\it broadcast independence number} $\alpha_b(G)$ of $G$ is the maximum weight of an independent broadcast on $G$,
and an independent broadcast on $G$ 
of weight $\alpha_b(G)$ is {\it optimal}.
For an integer $k$, 
let $[k]$ be the set of all positive integers at most $k$.

Let $G$ be a connected graph.
A function $f$ that assigns $1$ to every vertex in some independent set in $G$, and $0$ to every other vertex of $G$,
is an independent broadcast on $G$,
which implies $\alpha_b(G)\geq \alpha(G)$.
Our main result in \cite{bera2} implies $\alpha_b(G)\leq 4\alpha(G)$,
and, hence, 
$$1\leq \frac{\alpha_b(G)}{\alpha(G)}\leq 4
\mbox{ for every connected graph $G$.}$$
The existing results and proofs 
suggest that $\frac{\alpha_b(G)}{\alpha(G)}$ 
should be smaller than $4$ 
for connected graphs $G$ 
of sufficiently large local expansion and sparsity.
Natural hypotheses ensuring these properties 
are lower bounds on the girth and the minimum degree.
In the present paper, we explore how much the upper bound
on $\frac{\alpha_b(G)}{\alpha(G)}$ can be improved
for connected graphs $G$ of large girth and minimum degree.
Our two main results are the following.

\begin{theorem}\label{theorem1}
If $G$ is a connected graph of girth at least $6$ 
and minimum degree at least $3$, then
$$\alpha_b(G)<2\alpha(G).$$
\end{theorem}

\begin{theorem}\label{theorem2}
For every positive integer $k$,
there is a connected graph $G$ of girth at least $k$ and minimum degree at least $k$ 
such that 
$$\alpha_b(G)\geq 2\left(1-\frac{1}{k}\right)\alpha(G).$$ 
\end{theorem}
Together, these two results imply that 
lower bounds on the girth and the minimum degree
of a connected graph $G$
can lower the fraction $\frac{\alpha_b(G)}{\alpha(G)}$
from $4$ below $2$, but not any further.
The proof of Theorem \ref{theorem2}
is an adaptation of Erd\H{o}s's \cite{e} 
famous probabilistic proof 
of the existence of graphs of arbitrarily 
large girth and chromatic number,
and it actually implies the existence, 
for every positive integer $k$,
of a connected graph $G$ of girth at least $k$ 
and minimum degree at least $k$ 
such that 
$$\rho(G)\geq \left(1-\frac{1}{k}\right)\alpha(G).$$
The method used in the proof of Theorem \ref{theorem1}
also yields the following.

\begin{theorem}\label{theorem3}
Let $G$ be a connected graph of girth at least $g$ 
and minimum degree at least $\delta$.
\begin{enumerate}[(i)]
\item If $g=6$ and $\delta=5$, then
$\alpha_b(G) \leq \alpha(G)+\rho(G).$
\item If $\xi$ is a real number with $2\leq \xi<4$,
$g=4$, and $\delta\geq \frac{10}{\xi}$, then
$\alpha_b(G) \leq \xi\alpha(G).$
\end{enumerate}
\end{theorem}
All proofs are given in the next section.

\section{Proofs}

\begin{proof}[Proof of Theorem \ref{theorem1}]
Let $G$ be as in the statement.
Let $f:V(G)\to\mathbb{N}_0$ be an optimal independent broadcast on $G$.
Let $X=\{ x\in V(G):f(x)>0\}$.
To every vertex $x$ in $X$, we assign a set $I(x)$ as follows:
\begin{itemize}
\item If $1\leq f(x)\leq 2$, then let $I(x)=\{ x\}$.
\item If $3\leq f(x)\leq 5$, then let $I(x)=N_G(x)$.
\item If $6\leq f(x)\leq 13$, then let 
$I(x)=\big\{ y\in V(G):{\rm dist}_G(x,y)\in \{ 0,2\}\big\}.$
\item If $f(x)\geq 14$, then, by (B1),
there is a shortest path $P(x):xx_1\ldots x_{2\ell+4}$ 
in $G$ with $\ell=\left\lfloor\frac{f(x)-9}{4}\right\rfloor$.
Let
$$I(x)=\big\{ y\in V(G):{\rm dist}_G(x,y)\in \{ 0,2\}\big\}
\cup\bigcup\limits_{i=1}^\ell \big(N_G(x_{2i+3})\setminus \{ x_{2i+2}\}\big).$$ 
See Figure \ref{fig1} for an illustration.
\end{itemize}
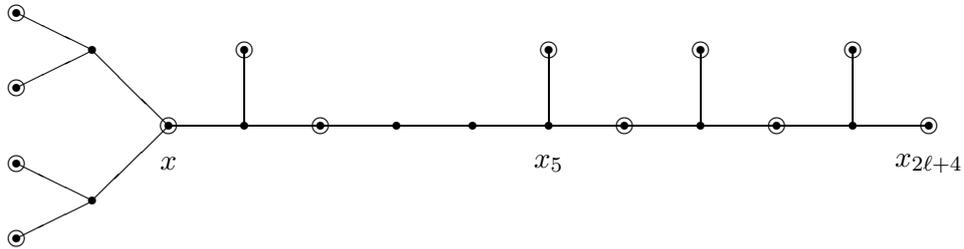
\begin{figure}[H]
\begin{center}
\unitlength 1mm 
\linethickness{0.4pt}
\ifx\plotpoint\undefined\newsavebox{\plotpoint}\fi 
\begin{picture}(121,31)(0,0)
\put(20,15){\circle*{1}}
\put(30,15){\circle*{1}}
\put(40,15){\circle*{1}}
\put(30,25){\circle*{1}}
\put(0,20){\circle*{1}}
\put(0,0){\circle*{1}}
\put(20,15){\line(1,0){20}}
\put(50,15){\circle*{1}}
\put(60,15){\circle*{1}}
\put(70,15){\circle*{1}}
\put(80,15){\circle*{1}}
\put(90,15){\circle*{1}}
\put(100,15){\circle*{1}}
\put(110,15){\circle*{1}}
\put(120,15){\circle*{1}}
\put(120,15){\line(-1,0){80}}
\put(70,25){\circle*{1}}
\put(90,25){\circle*{1}}
\put(110,25){\circle*{1}}
\put(110,25){\line(0,-1){10}}
\put(90,25){\line(0,-1){10}}
\put(70,25){\line(0,-1){10}}
\put(10,25){\circle*{1}}
\put(10,5){\circle*{1}}
\put(0,30){\circle*{1}}
\put(0,10){\circle*{1}}
\put(0,30){\circle{2}}
\put(0,10){\circle{2}}
\put(30,25){\circle{2}}
\put(0,20){\circle{2}}
\put(0,0){\circle{2}}
\put(40,15){\circle{2}}
\put(80,15){\circle{2}}
\put(70,25){\circle{2}}
\put(20,15){\circle{2}}
\put(90,25){\circle{2}}
\put(100,15){\circle{2}}
\put(110,25){\circle{2}}
\put(120,15){\circle{2}}
\put(20,10){\makebox(0,0)[cc]{$x$}}
\put(120,10){\makebox(0,0)[cc]{$x_{2\ell+4}$}}
\put(0,30){\line(2,-1){10}}
\put(10,25){\line(1,-1){10}}
\put(20,15){\line(-1,-1){10}}
\put(10,5){\line(-2,1){10}}
\put(0,20){\line(2,1){10}}
\put(30,15){\line(0,1){10}}
\put(10,5){\line(-2,-1){10}}
\put(70,10){\makebox(0,0)[cc]{$x_5$}}
\end{picture}
\end{center}
\caption{The set $I(x)$ for a vertex $x$ with $f(x)\in \{ 21,22,23,24\}$, where we assume that certain vertices have degree exactly $3$.}\label{fig1}
\end{figure}
By the girth condition and the choice of $P(x)$ as a shortest path,
the set $I(x)$ is an independent set for every $x$ in $X$.

Suppose, for a contradiction, 
that there are distinct vertices $x$ and $x'$ in $X$
such that the sets $I(x)$ and $I(x')$ 
intersect or are joined by an edge.
Let $f(x)\geq f(x')$.
If $1\leq f(x)\leq 2$, then ${\rm dist}_G(x,x')=1$,
if $3\leq f(x)\leq 5$, then ${\rm dist}_G(x,x')\leq 3$, and
if $6\leq f(x)\leq 13$, then ${\rm dist}_G(x,x')\leq 5$, 
which contradicts (B2) in each case.
Now, let $f(x)\geq 14$.
If $f(x')\leq 13$, then 
$${\rm dist}_G(x,x')\leq 
\left(2\left\lfloor\frac{f(x)-9}{4}\right\rfloor+4\right)+3
\leq 
\frac{f(x)-9}{2}+7
\leq f(x),$$
and, 
if $f(x')\geq 14$, then 
\begin{eqnarray*}
{\rm dist}_G(x,x')&\leq &
\left(2\left\lfloor\frac{f(x)-9}{4}\right\rfloor+4\right)+1+\left(2\left\lfloor\frac{f(x')-9}{4}\right\rfloor+4\right)\\
&\leq &
\frac{f(x)}{2}+\frac{f(x')}{2}\\
&\leq & \max\{ f(x),f(x')\},
\end{eqnarray*}
again contradicting (B2) in each case.
Therefore, $I=\bigcup\limits_{x\in X}I(x)$ is an independent set in $G$.

Let $x$ be a vertex in $X$.
If either $f(x)=1$ or $3\leq f(x)\leq 13$, 
then the girth and degree conditions imply 
$|I(x)|>\frac{f(x)}{2}$.
Similarly, if $f(x)\geq 14$, 
then, by the girth and degree conditions,
and the choice of $P(x)$ as a shortest path,
we obtain
$$|I(x)|\geq 7+2\left\lfloor\frac{f(x)-9}{4}\right\rfloor
\geq 7+\frac{f(x)-12}{2}>\frac{f(x)}{2}.$$
Finally, if $f(x)=2$, then $|I(x)|=\frac{f(x)}{2}$,
that is, only in this final case, equality holds.

Altogether, we obtain
$$\alpha(G)\geq |I|
\geq \sum\limits_{x\in X}|I(x)|
\geq \sum\limits_{x\in X}\frac{f(x)}{2}
\geq \frac{\alpha_b(G)}{2}.$$
Suppose, for a contradiction, 
that $\alpha(G)=\frac{\alpha_b(G)}{2}$,
that is, the above inequality chain holds with equality throughout.
This implies that $f(x)=2$ for every $x$ in $X$.
By (B2), the set $X$ is a packing in $G$, 
which implies 
$$\alpha(G)\geq \rho(G)\geq |X|=\frac{\alpha_b(G)}{2}=\alpha(G),$$
that is, $\alpha(G)=\rho(G)$, and $X$ is a maximum packing in $G$.
Now, replacing $x$ within $X$ by two nonadjacent neighbors yields
an independent set of order $|X|+1$,
contradicting $\alpha(G)=\rho(G)$;
cf. \cite{jora} for a structural characterization 
of the graphs that satisfy $\alpha(G)=\rho(G)$.
This completes the proof.
\end{proof}

\begin{proof}[Proof of Theorem \ref{theorem2}]
Let $k$ be a fixed integer at least $3$.
Let the real $\epsilon$ be such that $0<\epsilon<\frac{1}{k^2}$.
Let $H$ be a random graph in ${\cal G}(n,p)$ for $p=n^{\epsilon-1}$.
Let $V(H)=\{ u_1,\ldots,u_n\}$.
Let $G$ arise from the disjoint union of $n$ copies 
$S_1,\ldots,S_n$ of the star $K_{1,k}$ of order $k+1$,
where $S_i$ has center vertex $c_i$ 
and set of endvertices $L_i$ for $i$ in $[n]$, as follows:
For every edge $u_iu_j$ of $H$,
select one vertex $x_i$ in $L_i$ uniformly at random
and one vertex $x_j$ in $L_j$ uniformly at random,
and add the edge $x_ix_j$ to $G$.

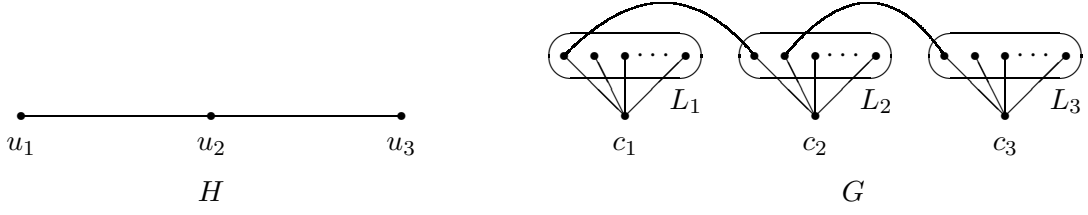
\begin{figure}
\begin{center}
$\mbox{}$\hfill
\unitlength 1mm 
\linethickness{0.4pt}
\ifx\plotpoint\undefined\newsavebox{\plotpoint}\fi 
\begin{picture}(56,6)(0,0)
\put(5,5){\circle*{1}}
\put(55,5){\circle*{1}}
\put(30,5){\circle*{1}}
\put(5,5){\line(1,0){50}}
\put(5,1){\makebox(0,0)[cc]{$u_1$}}
\put(30,1){\makebox(0,0)[cc]{$u_2$}}
\put(30,-5){\makebox(0,0)[cc]{$H$}}
\put(55,1){\makebox(0,0)[cc]{$u_3$}}
\end{picture}\hfill
\unitlength 1mm 
\linethickness{0.4pt}
\ifx\plotpoint\undefined\newsavebox{\plotpoint}\fi 
\begin{picture}(75,27)(0,0)
\put(15,5){\circle*{1}}
\put(65,5){\circle*{1}}
\put(40,5){\circle*{1}}
\put(15,1){\makebox(0,0)[cc]{$c_1$}}
\put(40,1){\makebox(0,0)[cc]{$c_2$}}
\put(65,1){\makebox(0,0)[cc]{$c_3$}}
\put(40,13){\oval(20,6)[]}
\put(65,13){\oval(20,6)[]}
\put(15,13){\oval(20,6)[]}
\put(7,13){\circle*{1}}
\put(32,13){\circle*{1}}
\put(57,13){\circle*{1}}
\put(11,13){\circle*{1}}
\put(36,13){\circle*{1}}
\put(61,13){\circle*{1}}
\put(15,13){\circle*{1}}
\put(40,13){\circle*{1}}
\put(65,13){\circle*{1}}
\put(23,13){\circle*{1}}
\put(48,13){\circle*{1}}
\put(73,13){\circle*{1}}
\put(19,13){\makebox(0,0)[cc]{$\cdots$}}
\put(44,13){\makebox(0,0)[cc]{$\cdots$}}
\put(69,13){\makebox(0,0)[cc]{$\cdots$}}
\put(7,13){\line(1,-1){8}}
\put(32,13){\line(1,-1){8}}
\put(57,13){\line(1,-1){8}}
\put(15,5){\line(-1,2){4}}
\put(40,5){\line(-1,2){4}}
\put(65,5){\line(-1,2){4}}
\put(15,13){\line(0,-1){8}}
\put(40,13){\line(0,-1){8}}
\put(65,13){\line(0,-1){8}}
\put(15,5){\line(1,1){8}}
\put(40,5){\line(1,1){8}}
\put(65,5){\line(1,1){8}}
\qbezier(7,13)(20.5,27)(32,13)
\qbezier(36,13)(47.5,27)(57,13)
\put(23,7){\makebox(0,0)[cc]{$L_1$}}
\put(48,7){\makebox(0,0)[cc]{$L_2$}}
\put(45,-5){\makebox(0,0)[cc]{$G$}}
\put(73,7){\makebox(0,0)[cc]{$L_3$}}
\end{picture}\hfill$\mbox{}$
\end{center}
\caption{Some $H$ and $G$.}\label{fig3}
\end{figure}
If $X$ denotes the number of cycles of length less than $k$ in $H$,
then it is known (cf. Theorem 11.2.2. in \cite{di}) that 
$$\lim\limits_{n\to\infty}\mathbb{P}\left[X\geq \frac{n}{2}\right]=0.$$
A set $I$ is an {\it independent transversal} if 
\begin{enumerate}[(i)]
\item $I$ is an independent set in $G$,
\item $I\cap \{ c_1,\ldots,c_n\}=\emptyset$, and 
\item $|I\cap L_i|\leq 1$ for every $i$ in $[n]$.
\end{enumerate}
Note that if $i$ and $j$ are distinct indices in $[n]$,
then a vertex in $L_i$ is adjacent to a vertex in $L_j$
with probability $\frac{p}{k^2}$.
Note furthermore, that there are ${n\choose r}k^r$
sets $I$ of order $r$ that satisfy the conditions (ii) and (iii) above.
Therefore, if $\beta$ denotes the maximum order of an independent transversal, then, by the union bound, 
we obtain, for $r=\frac{n}{2k^2}$,
\begin{eqnarray*}
\mathbb{P}\left[\beta\geq r\right] & \leq & {n\choose r}k^r\left(1-\frac{p}{k^2}\right)^{{r \choose 2}}\\
& \leq & n^rk^r\left(1-\frac{p}{k^2}\right)^{r(r-1)/2}\\
& = & \left(nk\left(1-\frac{p}{k^2}\right)^{(r-1)/2}\right)^r\\
& \leq & \left(nke^{-\frac{p(r-1)}{2k^2}}\right)^r
\,\,\,\,\,\,\,\mbox{(using $1-x\leq e^{-x}$).}
\end{eqnarray*}
For $n$ sufficiently large, 
we have $p\geq \frac{6k^4\ln n}{n}$, 
which implies (cf. Lemma 11.2.1. in \cite{di})
\begin{eqnarray*}
nke^{-\frac{p(r-1)}{2k^2}} 
& = & nke^{\left(-\frac{pn}{4k^4}+\frac{p}{2k^2}\right)}
\leq nke^{\left(-\frac{3}{2}\ln (n)+\frac{1}{2}\right)}
= \frac{k\sqrt{e}}{\sqrt{n}}\to 0\mbox{ for }n\to\infty,
\end{eqnarray*}
and, hence, 
$$\lim\limits_{n\to\infty}\mathbb{P}\left[\beta\geq \frac{n}{2k^2}\right]=0.$$
Therefore, if $n$ is sufficiently large, 
then 
$$\mathbb{P}\left[X\geq \frac{n}{2}\right]+\mathbb{P}\left[
\beta\geq \frac{n}{2k^2}\right]<1,$$
which implies the existence of a graph $H$ in ${\cal G}(n,p)$,
and a graph $G$ as above such that 
$X<\frac{n}{2}$
and 
$\beta<\frac{n}{2k^2}$.

For an induced subgraph $H'$ of $H$, 
let $G(H')=G\left[\bigcup\limits_{u_i\in V(H')}V(S_i)\right]$.

Let $F$ be a set of at most $\frac{n}{2}$ vertices of $H$
such that $H_0=H-F$ has no cycle of length less than $k$.
By construction, the graph $G(H_0)$ 
has no cycle of length less than $k$.
Note that $H_0$ has order at least $\frac{n}{2}$.

We construct a finite sequence $H_0,\ldots,H_\ell$ as follows: 
Let $i$ be a nonnegative integer such that $H_i$ is defined.
If $G(H_i)$ has minimum degree at least $k$,
then let $\ell=i$, and terminate the sequence.
Otherwise, $G(H_i)$ has a vertex $x_i$ of degree less than $k$.
By construction, there is a vertex $u_s$ of $H_i$ with $x_i\in L_s$.
Let $N$ be the set of indices $j$ in $[n]$ 
such that $x_i$ has a neighbor in $L_j$,
and let $H_{i+1}=H_i-\{ u_s\}\cup \{ u_j:j\in N\}$.
Note that $|N|<k$.

Since $\{ x_1,\ldots,x_{\ell}\}$ is an independent transversal,
we have $\ell\leq \frac{n}{2k^2}$, which implies that 
$H_\ell$ has order $n_{\ell}$ at least
$\frac{n}{2}-\frac{nk}{2k^2}=\frac{n}{2}\left(1-\frac{1}{k}\right)$.
The graph $G(H_{\ell})$ 
has girth at least $k$,
minimum degree at least $k$, 
and no independent transversal of order $\frac{n}{2k^2}$.
If $G(H_{\ell})$ is disconnected,
then adding some bridges to $G(H_{\ell})$
between different sets $L_i$ yields a connected graph $G^*$ that 
has girth at least $k$,
minimum degree at least $k$, 
and no independent transversal of order $\frac{n}{2k^2}$.

The function $f:V(G^*)\to\mathbb{N}_0$
that assigns $2$ to every vertex in $\{ c_i:u_i\in V(H_{\ell})\}$,
and $0$ to every other vertex,
is an independent broadcast on $G^*$,
which implies 
$\alpha_b(G^*)\geq 2n_{\ell}$.
Now, let $J$ be a maximum independent set in $G^*$.
Since $G^*$ has no independent transversal of order $\frac{n}{2k^2}$,
there are less than $\frac{n}{2k^2}$ indices $i$ in $[n]$
such that $J$ intersects $L_i$, which implies
$\alpha(G^*)=|J|\leq n_{\ell}+\frac{nk}{2k^2}=n_{\ell}+\frac{n}{2k}$.
Now,
$$\frac{\alpha_b(G^*)}{\alpha(G^*)} \geq 
\frac{2n_{\ell}}{n_{\ell}+\frac{n}{2k}}
\geq
\frac{2\frac{n}{2}\left(1-\frac{1}{k}\right)}{\frac{n}{2}\left(1-\frac{1}{k}\right)+\frac{n}{2k}}
=2\left(1-\frac{1}{k}\right),
$$
which completes the proof.
\end{proof}

\begin{proof}[Proof of Theorem \ref{theorem3}]
Let $G$ be a connected graph of girth at least $g$ 
and minimum degree at least $\delta$.
Let $f:V(G)\to\mathbb{N}_0$ be an optimal independent broadcast on $G$.
Let $X=\{ x\in V(G):f(x)>0\}$.

\bigskip

\noindent (i) First, we assume that $g=6$ and $\delta=5$.

To every vertex $x$ in $X$, 
we assign a set $I(x)$ as follows:
\begin{itemize}
\item If $1\leq f(x)\leq 2$, then let $I(x)=\{ x\}$.
\item If $f(x)\geq 3$, then, by (B1),
there is a shortest path $P(x):xx_1\ldots x_{2\ell-1}$ in $G$ with $\ell=\left\lfloor\frac{f(x)+1}{4}\right\rfloor$.
Let
$$I(x)=N_G(x)\cup\bigcup\limits_{i=2}^\ell \big(N_G(x_{2i-2})\setminus \{ x_{2i-3}\}\big).$$ 
See Figure \ref{fig2} for an illustration.
\end{itemize}

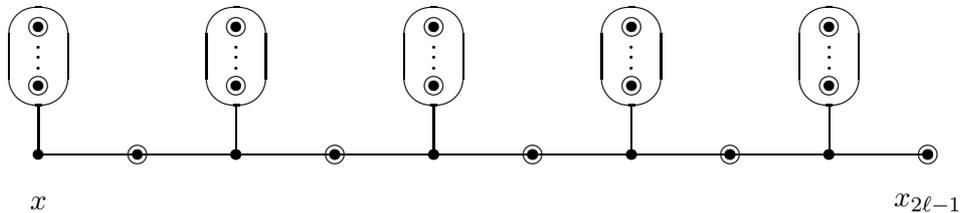
\begin{figure}[H]
\begin{center}
\unitlength 1.3mm 
\linethickness{0.4pt}
\ifx\plotpoint\undefined\newsavebox{\plotpoint}\fi 
\begin{picture}(97,20)(0,0)
\put(5,5){\circle*{1}}
\put(25,5){\circle*{1}}
\put(45,5){\circle*{1}}
\put(65,5){\circle*{1}}
\put(85,5){\circle*{1}}
\put(15,5){\circle*{1}}
\put(35,5){\circle*{1}}
\put(55,5){\circle*{1}}
\put(75,5){\circle*{1}}
\put(95,5){\circle*{1}}
\put(25,5){\circle*{1}}
\put(5,5){\line(1,0){20}}
\put(35,5){\circle*{1}}
\put(45,5){\circle*{1}}
\put(65,5){\circle*{1}}
\put(75,5){\circle*{1}}
\put(85,5){\circle*{1}}
\put(95,5){\circle*{1}}
\put(5,0){\makebox(0,0)[cc]{$x$}}
\put(55,5){\circle*{1}}
\put(5,15){\oval(6,10)[]}
\put(25,15){\oval(6,10)[]}
\put(45,15){\oval(6,10)[]}
\put(65,15){\oval(6,10)[]}
\put(85,15){\oval(6,10)[]}
\put(5,5){\line(0,1){5}}
\put(25,5){\line(0,1){5}}
\put(45,5){\line(0,1){5}}
\put(65,5){\line(0,1){5}}
\put(85,5){\line(0,1){5}}

\put(5,12){\circle*{1}}
\put(25,12){\circle*{1}}
\put(45,12){\circle*{1}}
\put(65,12){\circle*{1}}
\put(85,12){\circle*{1}}
\put(5,12){\circle{2}}
\put(25,12){\circle{2}}
\put(45,12){\circle{2}}
\put(65,12){\circle{2}}
\put(85,12){\circle{2}}

\put(5,15.7){\makebox(0,0)[cc]{$\vdots$}}
\put(25,15.7){\makebox(0,0)[cc]{$\vdots$}}
\put(45,15.7){\makebox(0,0)[cc]{$\vdots$}}
\put(65,15.7){\makebox(0,0)[cc]{$\vdots$}}
\put(85,15.7){\makebox(0,0)[cc]{$\vdots$}}

\put(5,18){\circle*{1}}
\put(25,18){\circle*{1}}
\put(45,18){\circle*{1}}
\put(65,18){\circle*{1}}
\put(85,18){\circle*{1}}
\put(5,18){\circle{2}}
\put(25,18){\circle{2}}
\put(45,18){\circle{2}}
\put(65,18){\circle{2}}
\put(85,18){\circle{2}}

\put(15,5){\circle{2}}
\put(35,5){\circle{2}}
\put(55,5){\circle{2}}
\put(75,5){\circle{2}}
\put(95,5){\circle{2}}
\put(25,5){\line(1,0){70}}
\put(95,0){\makebox(0,0)[cc]{$x_{2\ell-1}$}}
\end{picture}
\end{center}
\caption{The set $I(x)$ for a vertex $x$ with $f(x)\in \{ 19,20,21,22\}$.}\label{fig2}
\end{figure}
It follows similarly as in the proof of Theorem \ref{theorem1}
that the $I(x)$ are disjoint independent sets in $G$
that are not joined by edges within $G$.

Let $x$ be a vertex in $X$.
If $f(x)=1$, then $|I(x)|=f(x)$,
if $f(x)=2$, then $|I(x)|=f(x)-1$, and,
if $f(x)\geq 3$,
then, by the girth and degree conditions 
and the choice of $P(x)$ as a shortest path,
$$|I(x)|\geq 5+4\left(\left\lfloor\frac{f(x)+1}{4}\right\rfloor-1\right)
\geq 5+4\left(\frac{f(x)-2}{4}-1\right)
=f(x)-1.$$
Let $X_1=\{ x\in V(G):f(x)=1\}$.
It follows that $I=\bigcup\limits_{x\in X}I(x)$ is an independent set in $G$ of order at least
$\alpha_b(G)-|X\setminus X_1|=\sum\limits_{x\in X_1}f(x)
+\sum\limits_{x\in X\setminus X_1}(f(x)-1)$.
Since $X\setminus X_1$ is a packing in $G$,
we obtain
$\alpha(G)\geq \alpha_b(G)-|X\setminus X_1|
\geq \alpha_b(G)-\rho(G)$,
which completes the proof of (i).

\bigskip

\noindent (ii) Next, we assume that 
$\xi$ is a real number with $2\leq \xi<4$,
$g=4$, and $\delta\geq \frac{10}{\xi}$.

To every vertex $x$ in $X$, 
we assign a set $I(x)$ as follows:
\begin{itemize}
\item If $1\leq f(xleq 2$, then let $I(x)=\{ x\}$.
\item If $f(x)\geq 3$, then, by (B1),
there is a shortest path $P(x):xx_1\ldots x_{4\ell-3}$ in $G$ with $\ell=\left\lfloor\frac{f(x)+5}{8}\right\rfloor$.
Let $x_0=x$, and let
$$I(x)=\bigcup\limits_{i=1}^\ell N_G(x_{4(i-1)}).$$ 
See Figure \ref{fig4} for an illustration.
\end{itemize}
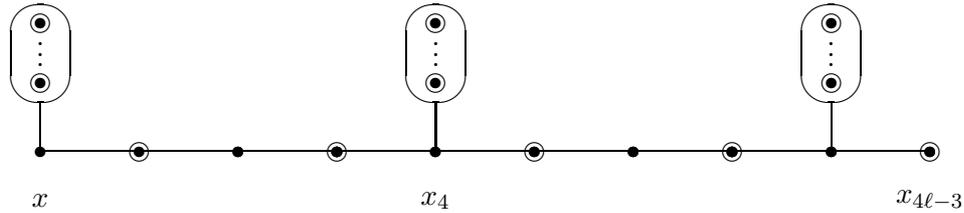
\begin{figure}[H]
\begin{center}
\unitlength 1.3mm 
\linethickness{0.4pt}
\ifx\plotpoint\undefined\newsavebox{\plotpoint}\fi 
\begin{picture}(96,20)(0,0)
\put(5,5){\circle*{1}}
\put(25,5){\circle*{1}}
\put(65,5){\circle*{1}}
\put(45,5){\circle*{1}}
\put(85,5){\circle*{1}}
\put(15,5){\circle*{1}}
\put(35,5){\circle*{1}}
\put(75,5){\circle*{1}}
\put(35,5){\circle*{1}}
\put(75,5){\circle*{1}}
\put(55,5){\circle*{1}}
\put(95,5){\circle*{1}}
\put(35,5){\circle*{1}}
\put(75,5){\circle*{1}}
\put(25,5){\circle*{1}}
\put(65,5){\circle*{1}}
\put(45,5){\circle*{1}}
\put(85,5){\circle*{1}}
\put(5,5){\line(1,0){20}}
\put(45,5){\line(1,0){20}}
\put(35,5){\circle*{1}}
\put(75,5){\circle*{1}}
\put(55,5){\circle*{1}}
\put(95,5){\circle*{1}}
\put(5,0){\makebox(0,0)[cc]{$x$}}
\put(5,15){\oval(6,10)[]}
\put(45,15){\oval(6,10)[]}
\put(85,15){\oval(6,10)[]}
\put(5,5){\line(0,1){5}}
\put(45,5){\line(0,1){5}}
\put(85,5){\line(0,1){5}}
\put(5,12){\circle*{1}}
\put(45,12){\circle*{1}}
\put(85,12){\circle*{1}}
\put(5,12){\circle{2}}
\put(45,12){\circle{2}}
\put(85,12){\circle{2}}
\put(5,15.7){\makebox(0,0)[cc]{$\vdots$}}
\put(45,15.7){\makebox(0,0)[cc]{$\vdots$}}
\put(85,15.7){\makebox(0,0)[cc]{$\vdots$}}
\put(5,18){\circle*{1}}
\put(45,18){\circle*{1}}
\put(85,18){\circle*{1}}
\put(5,18){\circle{2}}
\put(45,18){\circle{2}}
\put(85,18){\circle{2}}
\put(15,5){\circle{2}}
\put(55,5){\circle{2}}
\put(95,5){\circle{2}}
\put(35,5){\circle{2}}
\put(75,5){\circle{2}}
\put(95,0){\makebox(0,0)[cc]{$x_{4\ell-3}$}}
\put(25,5){\line(1,0){10}}
\put(65,5){\line(1,0){10}}
\put(45,5){\line(1,0){10}}
\put(85,5){\line(1,0){10}}
\put(35,5){\line(1,0){10}}
\put(75,5){\line(1,0){10}}
\put(45,0){\makebox(0,0)[cc]{$x_4$}}
\end{picture}
\end{center}
\caption{The set $I(x)$ for a vertex $x$ with $f(x)\in \{ 19,\ldots,26\}$.}\label{fig4}
\end{figure}
Again,
the $I(x)$ are disjoint independent sets in $G$
that are not joined by edges within $G$.

Let $x$ be a vertex in $X$.
If $1\leq f(x)\leq 2$, then $|I(x)|\geq \frac{f(x)}{2}\geq \frac{f(x)}{\xi}$,
if $3\leq f(x)\leq \left\lfloor \xi \delta\right\rfloor$, 
then $|I(x)|\geq \delta\geq \frac{f(x)}{\xi}$,
and, if $f(x)\geq \left\lfloor \xi \delta\right\rfloor+1$
then, by the girth and degree conditions 
and the choice of $P(x)$ as a shortest path,
$$|I(x)|\geq \delta\left\lfloor\frac{f(x)+5}{8}\right\rfloor
\geq \delta\frac{f(x)-2}{8}
\geq \frac{f(x)}{\xi},$$
where we use $f(x)\geq \xi \delta$ 
and $\delta\geq \frac{10}{\xi}$.
It follows that $\alpha(G)\geq \frac{\alpha_b(G)}{\xi}$,
which completes the proof of (ii).
\end{proof}


\begin{thebibliography}{}
\bibitem{ahboso} M. Ahmane, I. Bouchemakh, E. Sopena, On the broadcast independence number of caterpillars, Discrete Applied Mathematics 244 (2018) 20-35.
\bibitem{bera1} S. Bessy, D. Rautenbach, Algorithmic aspects of broadcast independence, arXiv 1809.07248.
\bibitem{bera2} S. Bessy, D. Rautenbach, Relating broadcast independence and independence, manuscript 2018.
\bibitem{boze} I. Bouchemakh, M. Zemir, On the broadcast independence number of grid graph, Graphs and Combinatorics 30 (2014) 83-100.
\bibitem{di} R. Diestel, Graph theory. 2nd ed., Graduate Texts in Mathematics. 173. Berlin: Springer. xiv, 313 p. (2000). 
\bibitem{duerhahehe} J.E. Dunbar, D.J. Erwin, T.W. Haynes, S.M. Hedetniemi, S.T. Hedetniemi, Broadcasts in graphs, Discrete Applied Mathematics 154 (2006) 59-75.
\bibitem{e} P. Erd\H{o}s, Graph theory and probability II, Canadian Journal of Mathematics 13 (1961) 346-352.
\bibitem{er} D.J. Erwin, Cost domination in graphs, (Ph.D. thesis), Western Michigan University, 2001.
\bibitem{jora} F. Joos, D. Rautenbach, Equality of distance packing numbers, Discrete Mathematics 338 (2015) 2374-2377.
\bibitem{tovo} J. Topp, L. Volkmann, On packing and covering numbers of graphs, Discrete Mathematics 96 (1991) 229-238.
\end{thebibliography}
\end{document}